\documentclass[12pt]{article}
\usepackage{amsmath,amssymb,amsthm,amsfonts,mathrsfs}
\usepackage{appendix}
\usepackage{array,enumitem,graphics,xcolor,yfonts,colonequals}
\usepackage{stmaryrd}
\usepackage[alphabetic]{amsrefs}
\usepackage[all,cmtip]{xy}
\usepackage{tikz-cd}
\usepackage[colorlinks,anchorcolor=blue,citecolor=blue,linkcolor=blue,urlcolor=blue,bookmarksopen=true]{hyperref}
\usepackage{comment}
\usepackage{mathtools}
\usepackage[margin=1in]{geometry}

\usepackage{titlesec}
\titleformat{\section}
  {\normalfont\large\bf}{\thesection}{1em}{}
\titleformat{\subsection}[runin]
  {\normalfont\normalsize\bf}{\thesubsection}{1em}{}

\usepackage{fancyhdr}
\fancypagestyle{titlepage}
{
	\fancyhf{}

	\fancyfoot[l]{
	\href{https://mathscinet.ams.org/mathscinet/msc/msc2020.html}
	}
}

\newcommand{\bC}{\mathbb{C}}    
\newcommand{\bR}{\mathbb{R}}    
\newcommand{\bZ}{\mathbb{Z}}    

\newcommand{\cD}{\mathcal{D}}
\newcommand{\cE}{\mathcal{E}}   
\newcommand{\cO}{\mathcal{O}}   

\newcommand{\sA}{\mathscr{A}}   
\newcommand{\sD}{\mathscr{D}}   

\newcommand{\Coh}{\mathrm{Coh}} 
\newcommand{\Db}{\mathrm{D^b}}  

\newcommand{\Aut}{\operatorname{Aut}}   
\newcommand{\Hom}{\operatorname{Hom}}   
\newcommand{\Pic}{\operatorname{Pic}}	

\newtheorem*{thm*}{Theorem}
\newtheorem*{prop*}{Proposition}
\newtheorem*{cor*}{Corollary}
\newtheorem{thm}{Theorem}[section]

\newtheorem{lemma}[thm]{Lemma}

\numberwithin{equation}{section}

\theoremstyle{definition}

\newtheorem{rmk}[thm]{Remark}

\begin{document}
\title{Shifting numbers of abelian varieties \\ via bounded $t$-structures}
\author{Yu-Wei Fan}
\date{}

\newcommand{\ContactInfo}{{
\bigskip\footnotesize

\bigskip
\noindent Y.-W.~Fan,
\textsc{Yau Mathematical Sciences Center\\
Tsinghua University\\
Beijing 100084, P. R. China}\par\nopagebreak
\noindent\textsc{Email:} \texttt{ywfan@mail.tsinghua.edu.cn}

}}

\maketitle
\thispagestyle{titlepage}

\begin{abstract}
The shifting numbers measure the asymptotic amount by which an endofunctor of a triangulated category translates inside the category, and are analogous to Poincar\'e translation numbers that are widely used in dynamical systems.
Motivated by this analogy, Fan--Filip raised the following question: ``Do the shifting numbers define a quasimorphism on the group of autoequivalences of a triangulated category?"
An affirmative answer was given by Fan--Filip for the bounded derived category of coherent sheaves on an elliptic curve or an abelian surface, via properties of the spaces of Bridgeland stability conditions on these categories.

We prove in this article that the question has an affirmative answer for abelian varieties of arbitrary dimensions, generalizing the result of Fan--Filip.
One of the key steps is to establish an alternative definition of the shifting numbers via  bounded $t$-structures on triangulated categories.
 In particular, the full package of a Bridgeland stability condition (a bounded $t$-structure, and a central charge on a charge lattice) is not necessary for the purpose of computing the shifting numbers.
\end{abstract}


\section{Introduction}
The notion of a $t$-structure arose in the work of Beilinson, Bernstein, and Deligne \cite{BBD} on perverse sheaves. A full additive subcategory $\sA$ of a triangulated category $\sD$ is the \emph{heart} of a bounded $t$-structure on $\sD$ if and only if the following two conditions hold:
\begin{enumerate}[label=(\alph*)]
\item $\Hom_\sD(A_1[k_1],A_2[k_2])=0$ if $k_1>k_2$ and $A_1,A_2$ are objects of $\sA$,
\item for every nonzero object $E$ of $\sD$, there is a (unique) sequence of exact triangles
$$
\xymatrix@C=.5em{
0_{} \ar@{=}[r] & E_0 \ar[rrrr] &&&& E_1 \ar[rrrr] \ar[dll] &&&& E_2
\ar[rr] \ar[dll] && \cdots \ar[rr] && E_{m-1}
\ar[rrrr] &&&& E_m \ar[dll] \ar@{=}[r] &  E_{} \\
&&& A_1[k_1] \ar@{-->}[ull] &&&& A_2[k_2] \ar@{-->}[ull] &&&&&&&& A_m[k_m] \ar@{-->}[ull] 
}
$$
with the properties that $k_1>\cdots>k_m$ and that $A_1,\ldots,A_m$ are objects of $\sA$.
\end{enumerate}
\noindent A bounded $t$-structure is uniquely determined by its heart, which allows us to omit the definition of bounded $t$-structures and work with the hearts only.
Let us denote the maximal and minimal degrees of nonzero cohomology of an object $E$ with respect to the heart $\sA$ as:
$$
\phi^+_\sA(E)\coloneqq k_1 \qquad \text{ and } \qquad \phi^-_\sA(E)\coloneqq k_m.
$$
We show that these extremal degrees can be used to define invariants of \emph{endofunctors} of  triangulated categories.

\begin{thm}[see Theorem~\ref{thm-2}]
\label{thm}
Let $F\colon\sD\rightarrow\sD$ be an endofunctor of a triangulated category $\sD$ with a split generator $G$, and let $\sA\subseteq\sD$ be the heart of a bounded $t$-structure on $\sD$.
Then the limit
$$
\lim_{n\rightarrow\infty}\frac{\phi^+_\sA(F^nG)}{n}
$$
exists and is independent of the choices of $G$ and $\sA$, therefore gives an invariant of $F$.
Moreover, if $\sD$ admits a Serre functor, then the limit
$$
\lim_{n\rightarrow\infty}\frac{\phi^-_\sA(F^nG)}{n}
$$
also exists and is independent of the choices of $G$ and $\sA$.
\end{thm}

Intuitively, these invariants measure how quickly do the objects $F^nG$  ``spread out" with respect to the heart $\sA$ as $n\rightarrow\infty$.
We will show that they coincide with the \emph{upper} and \emph{lower shifting numbers} $\tau^\pm(F)$ of  endofunctors \cite{EL,FF}. For basic properties and examples of $\tau^\pm(F)$, we refer the readers to \cite{FF}. The proof of Theorem~\ref{thm} is short, and the main ingredients already appeared in \cite{DHKK,EL,FF,KOT}.
We explain in Remark~\ref{rmk:simpler} that the theorem could provide a simpler way of computing the shifting numbers of endofunctors.

\begin{rmk}
It is proved in \cite[Theorem~2.2.6]{FF} that if $\sigma$ is a Bridgeland stability conditions on $\sD$, then the limit
$$
\lim_{n\rightarrow\infty}\frac{\phi^+_\sigma(F^nG)}{n}
$$
exists and coincides with the upper shifting number $\tau^+(F)$ (similar statement holds for $\phi^-_\sigma$ when $\sD$ admits a Serre functor). Here $\phi^+_\sigma(F^nG)$ denotes the largest phase of the $\sigma$-semistable factor of $F^nG$. Recall that a stability condition $\sigma=(Z,\sA)$ on $\sD$ consists of: a choice of the heart of a bounded $t$-structure $\sA$, a group homomorphism from $K_0(\sD)$ to a finite rank free abelian group $\Gamma$, and a group homomorphism $Z\colon\Gamma\rightarrow\bC$, which together satisfy certain axioms. In particular, one has
$$
\phi^+_\sA(F^nG)=-\lfloor-\phi^+_\sigma(F^nG)\rfloor-1.
$$
Therefore Theorem~\ref{thm} is a generalization of \cite[Theorem~2.2.6]{FF} which uses only the heart $\sA$ and, most importantly, does not require the existence of stability conditions associated with $\sA$.
\end{rmk}

\begin{rmk}
The shifting numbers of endofunctors can be considered as the categorical analogue of the \emph{Poincar\'e translation numbers} defined on $\text{Homeo}^+_\bZ(\bR)$, the central extension of the group of the orientation-preserving homeomorphisms of the circle $\bR/\bZ$ \cite[Section~1]{FF}.
These invariants also are closely related to the \emph{Lyapunov exponents} of the automorphisms of a \emph{shift of finite type} $(X,\sigma)$ \cite{Sch}.
\end{rmk}

Motivated by the analogy between the shifting numbers and the Poincar\'e translation numbers, it was asked in \cite{FF} that in what situations does the shifting number
$$
\tau\coloneqq\frac{\tau^++\tau^-}{2}\colon\Aut(\sD)\rightarrow\bR
$$
define a \emph{quasimorphism} on $\Aut(\sD)$, i.e.~there exists a constant $C>0$ such that $|\tau(F_1F_2)-\tau(F_1)-\tau(F_2)|<C$ holds for any $F_1,F_2\in\Aut(\sD)$.
Note that the existence of nontrivial quasimorphisms on a group has some algebraic consequences for the structures of the group, see for instance \cite{Kot}. 
It is proved in \cite{FF} that the shifting numbers give a quasimorphism on $\Aut(\sD)$ when $\sD$ is the bounded derived category of an elliptic curve or an abelian surface, where the proofs rely on certain global properties of the spaces of Bridgeland stability conditions on these categories. Such properties on the spaces of Bridgeland stability conditions (including the non-emptiness) are not known for abelian varieties of higher dimensions.

We prove the following result for abelian varieties of arbitrary dimensions, thanks to Theorem~\ref{thm}.

\begin{thm}[see Theorem~\ref{thm-3}]
\label{thm:abvar}
Let $X$ be an abelian variety over a field $\mathbf{k}$, and let $\sD=\Db(X)$ be the bounded derived category of coherent sheaves on $X$. Then the upper and lower shifting numbers coincide for any $F\in\Aut(\sD)$. Moreover, the map
$$
\tau=\tau^\pm\colon\Aut(\sD)\rightarrow\bR
$$
is a quasimorphism.
\end{thm}

\noindent\textbf{Conventions.}
Throughout this article, all triangulated categories are assumed to be $\bZ$-graded, linear over a base field $\mathbf{k}$, saturated (i.e.~admits a dg-enhancement which is smooth and proper), and of finite type (i.e.~the $\mathbf{k}$-vector space $\oplus_{k\in\bZ}\Hom_\sD(E,F[k])$ is finite-dimensional for any pair of objects $E,F$ in $\sD$). Endofunctors of triangulated categories are assumed to be $\mathbf{k}$-linear, triangulated, and not virtually zero (i.e.~any power is not the zero functor).  An abelian variety over $\mathbf{k}$ is a connected, smooth, and proper group scheme over $\mathbf{k}$. \\

\noindent\textbf{Acknowledgment.}
The author would like to thank Simion Filip for collaborating on \cite{FF}, on which this article is based.

\section{Shifting numbers via bounded $t$-structures}
Let us begin with recalling several notions of ``distances" between pairs of objects in a triangulated category.
Let $E_1,E_2$ be objects in a triangulated category $\sD$. The \emph{complexity function} of $E_2$ relative to $E_1$ {\cite[Definition~2.1]{DHKK}} is defined to be the function sending a real number $t$ to
$$
\delta_t(E_1,E_2)\coloneqq\inf
\left\{
	\sum_{\ell=1}^m e^{k_\ell t}\ \Big|\ 
	{\substack{
		0=A_0\to A_1\to\cdots\to A_m\cong  E_2\oplus F\text{ for some }F\in\text{Ob}(\cD), \\
		\text{where }\text{Cone}(A_{\ell-1}\to A_\ell)\cong E_1[k_\ell]\text{ for all }\ell
	}}
\right\}\in[0,\infty].
$$
Here, $\delta_t(E_1,E_2)=0$ if $E_2\cong0$, and $\delta_t(E_1,E_2)=\infty$ if $E_2$ does not lie in the thick triangulated subcategory generated by $E_1$.

The \emph{upper} and \emph{lower Ext-distances}  from $E_1$ to $E_2$ are defined to be
\begin{align*}
\epsilon^+(E_1,E_2)&\coloneqq\max\{k\in\bZ\mid \Hom(E_1,E_2[-k])\neq0\},\\
\epsilon^-(E_1,E_2)&\coloneqq\min\{k\in\bZ\mid \Hom(E_1,E_2[-k])\neq0\}.
\end{align*}

We say an object $G$ is a \emph{split generator} of the category $\sD$ if $\delta_t(G,E)<\infty$ for any object $E$ in $\sD$. The dynamical invariants of an endofunctor $F\colon\sD\rightarrow\sD$ can be defined via the distances between $G$ and $F^nG$ as $n\rightarrow\infty$. For instance, the \emph{categorical entropy function} of $F$ {\cite[Definition~2.5]{DHKK}} is defined to be the function sending a real number $t$ to
$$
h_t(F)\coloneqq\lim_{n\rightarrow\infty}\frac{\log\delta_t(G,F^nG)}{n}.
$$
It is proved in \cite[Lemma~2.6]{DHKK} that the limit exists and is independent of the choice of $G$.
The categorical entropy function is a real-valued convex function in $t$ \cite[Theorem~2.1.6]{FF}. As $t\rightarrow\pm\infty$, it has finite linear growth and the growth rates
$$
\tau^\pm(F)\coloneqq\lim_{t\rightarrow\pm\infty}\frac{h_t(F)}{t}\in\bR,
$$
called the \emph{upper} and \emph{lower shifting numbers}, coincide with the linear growth rates of the upper and lower Ext-distances:
$$
\tau^\pm(F)=\lim_{n\rightarrow\infty}\frac{\epsilon^\pm(G,F^nG)}{n}.
$$
See \cite[Proposition~6.13]{EL} or \cite[Theorem~2.1.7]{FF} for the proof.

\begin{thm}
\label{thm-2}
Let $F\colon\sD\rightarrow\sD$ be an endofunctor of a triangulated category $\sD$ with a split generator $G$, and let $\sA\subseteq\sD$ be the heart of a bounded $t$-structure on $\sD$.
Then the limit
$$
\lim_{n\rightarrow\infty}\frac{\phi^+_\sA(F^nG)}{n}
$$
exists and is independent of the choices of $G$ and $\sA$.
Moreover, if $\sD$ admits a Serre functor, then the limit
$$
\lim_{n\rightarrow\infty}\frac{\phi^-_\sA(F^nG)}{n}
$$
also exists and is independent of the choices of $G$ and $\sA$.
\end{thm}

\begin{proof}
Let $\sA$ be the heart of a bounded $t$-structure on $\sD$.
Since
$$
\Hom(G,F^nG[-\epsilon^+(G,F^nG)])\neq0,
$$
one has
$$
\phi^-_\sA(G)\leq\phi^+_\sA(F^nG)-\epsilon^+(G,F^nG).
$$
By dividing by $n$ and letting $n\rightarrow\infty$, one obtains
\begin{equation}
\label{eqn:+}
\tau^+(F)\leq\liminf_{n\rightarrow\infty}\frac{\phi^+_\sA(F^nG)}{n}.
\end{equation}
On the other hand, observe that for any sequence of exact triangles
$$
\xymatrix@C=.5em{
{} \ar@{}[r] & 0 \ar[rrrr] &&&& \star \ar[rrrr] \ar[dll] &&&& \star
\ar[rr] \ar[dll] && \cdots \ar[rr] && \star
\ar[rrrr] &&&& \star \ar[dll] \ar@{=}[r] &  F^nG\oplus\star \\
&&& G[k_1] \ar@{-->}[ull] &&&& G[k_2] \ar@{-->}[ull] &&&&&&&& G[k_m] \ar@{-->}[ull] 
}
$$
we have
$$
\phi^+_\sA(F^nG)\leq\phi^+_\sA(F^nG\oplus\star)\leq\max\{\phi^+_\sA(G[k_1]),\ldots,\phi^+_\sA(G[k_m])\}.
$$
Thus at least one of the $k_1,\ldots,k_m$ is not smaller than $\phi^+_\sA(F^nG)-\phi^+_\sA(G)$. Therefore
$$
\delta_t(G,F^nG)\geq e^{\left(\phi^+_\sA(F^nG)-\phi^+_\sA(G)\right)t}  \quad \text{ for any }t>0.
$$
Hence the categorical entropy function satisfies
$$
h_t(F)=\lim_{n\rightarrow\infty}\frac{\log\delta_t(G,F^nG)}{n}\geq t\cdot\limsup_{n\rightarrow\infty}\frac{\phi^+_\sA(F^nG)}{n} \quad (t>0).
$$
Thus one obtains a lower bound of the upper shifting number
$$
\tau^+(F)=\lim_{t\rightarrow\infty}\frac{h_t(F)}{t}\geq\limsup_{n\rightarrow\infty}\frac{\phi^+_\sA(F^nG)}{n}.
$$
Combining with (\ref{eqn:+}), one deduces that the limit
$$
\lim_{n\rightarrow\infty}\frac{\phi^+_\sA(F^nG)}{n}
$$
exists and coincides with $\tau^+(F)$. In particular, it does not depend on the choice of the split generator $G$ and the heart $\sA$. This proves the first part of Theorem~\ref{thm-2}.

Suppose $\sD$ admits a Serre functor $S\colon\sD\rightarrow\sD$. Then we have
$$
\Hom(F^nG[-\epsilon^-(G,F^nG)],SG)\neq0,
$$
which implies that
$$
\phi^-_\sA(F^nG)-\epsilon^-(G,F^nG)\leq\phi^+_\sA(SG).
$$
By dividing by $n$ and letting $n\rightarrow\infty$, one obtains
\begin{equation}
\label{eqn:-}
\tau^-(F)\geq\limsup_{n\rightarrow\infty}\frac{\phi^-_\sA(F^nG)}{n}.
\end{equation}
On the other hand, observe that for any sequence of exact triangles
$$
\xymatrix@C=.5em{
{} \ar@{}[r] & 0 \ar[rrrr] &&&& \star \ar[rrrr] \ar[dll] &&&& \star
\ar[rr] \ar[dll] && \cdots \ar[rr] && \star
\ar[rrrr] &&&& \star \ar[dll] \ar@{=}[r] &  F^nG\oplus\star \\
&&& G[k_1] \ar@{-->}[ull] &&&& G[k_2] \ar@{-->}[ull] &&&&&&&& G[k_m] \ar@{-->}[ull] 
}
$$
we have
$$
\phi^-_\sA(F^nG)\geq\phi^-_\sA(F^nG\oplus\star)\geq\min\{\phi^-_\sA(G[k_1]),\ldots,\phi^-_\sA(G[k_m])\}.
$$
Thus at least one of the $k_1,\ldots,k_m$ is not larger than $\phi^-_\sA(F^nG)-\phi^-_\sA(G)$.
Hence
$$
\delta_t(G,F^nG)\geq e^{\left(\phi^-_\sA(F^nG)-\phi^-_\sA(G)\right)t}  \quad \text{ for any }t<0.
$$
One then obtains that
$$
h_t(F)=\lim_{n\rightarrow\infty}\frac{\log\delta_t(G,F^nG)}{n}\geq t\cdot\liminf_{n\rightarrow\infty}\frac{\phi^-_\sA(F^nG)}{n} \quad (t<0).
$$
By dividing both sides by $t$ and taking $t\rightarrow-\infty$, we have
$$
\tau^-(F)=\lim_{t\rightarrow-\infty}\frac{h_t(F)}{t}\leq\liminf_{n\rightarrow\infty}\frac{\phi^-_\sA(F^nG)}{n}.
$$
Combining with (\ref{eqn:-}), one deduces that the limit
$$
\lim_{n\rightarrow\infty}\frac{\phi^-_\sA(F^nG)}{n}
$$
exists and coincides with $\tau^-(F)$. This concludes the proof.
\end{proof}

\begin{rmk}
\label{rmk:simpler}
We explain in this remark that Theorem~\ref{thm-2} could provide a simpler way of computing the shifting numbers of endofunctors comparing with previous methods. In order to compute $\tau^\pm(F)$, it usually is necessary to understand at least one of the following:
\begin{itemize}
\item $\delta_t(G,F^nG)$: usually not computable;
\item $\epsilon_t(G,F^nG)$: usually harder to compute than $\phi^\pm_\sA(F^nG)$;
\item $\phi^\pm_\sigma(F^nG)$: usually harder to compute than $\phi^\pm_\sA(F^nG)$; for instance, let $\sD=\Db\Coh(X)$ be the bounded derived category of coherent sheaves on a smooth projective variety of dimension at least two, then $\phi^\pm_{\Coh(X)}(F^nG)$ is usually easier to compute than $\phi^\pm_\sigma(F^nG)$ since there is no stability conditions with heart $\sA=\Coh(X)$.
\end{itemize}
In other words, the main advantage of computing $\tau^\pm(F)$ via Theorem~\ref{thm-2} is that one can freely choose the heart $\sA$ to compute $\phi^\pm_\sA(F^nG)$ without any restrictions on $\sA$.
\end{rmk}

\section{Shifting numbers of abelian varieties}
We prove the following theorem in this section.
\begin{thm}
\label{thm-3}
Let $X$ be an abelian variety of dimension $d$ over a field $\mathbf{k}$, and let $\sD=\Db(X)$ be the bounded derived category of coherent sheaves on $X$. Then the upper and lower shifting numbers coincide for any $F\in\Aut(\sD)$. Moreover, the map
$$
\tau=\tau^\pm\colon\Aut(\sD)\rightarrow\bR
$$
is a quasimorphism.
\end{thm}

\begin{lemma}
\label{lemma}
Let $F\in\Aut(\sD)$ be an autoequivalence of $\sD$.
There exists an integer $n_F\in\bZ$ such that
$$
n_F\leq\phi_\sA^-\big(F(E)\big)\leq\phi_\sA^+\big(F(E)\big)\leq n_F+3d
$$
holds for any $E\in\sA\coloneqq\Coh(X)$.
\end{lemma}
\begin{proof}
The equivalence $F\cong\Phi_\cE$ is a Fourier--Mukai transform, where the kernel $\cE$ is isomorphic to a sheaf up to a shift \cite[Proposition~3.2]{Orl}.
The lemma then follows from:
\begin{enumerate}[label=(\alph*)]
\item any coherent sheaf $\cE$ on $X\times X$ admits a locally free resolution of length at most $2d$;
\item the higher direct image of a coherent sheaf on $X\times X$ under the second projection $X\times X\xrightarrow{\pi_2} X$ vanishes at degree $>d=\dim(X)$.
\end{enumerate}
\end{proof}

\begin{proof}[Proof of Theorem~\ref{thm-3}]
Let $\cO(1)$ be a very ample line bundle on $X$.
Then
$$
G\coloneqq\bigoplus_{i=0}^d\cO(i)\in\sA
$$
is a split generator of $\sD$ by \cite[Theorem~4]{Orl2}.
By Theorem~\ref{thm-2} and Lemma~\ref{lemma}, for any $F\in\Aut(\sD)$ we have
$$
\tau^+(F)-\tau^-(F)=\lim_{n\rightarrow\infty}\frac{\phi_\sA^+(F^nG)-\phi_\sA^-(F^nG)}{n}\leq
\lim_{n\rightarrow\infty}\frac{3d}{n}=0.
$$
Thus the upper and lower shifting numbers coincide for any autoequivalence $F$.
Note that $\tau(F)\coloneqq\tau^\pm(F)$ is the homogenization of
$$
\widetilde{\tau}\colon\Aut(\sD)\rightarrow\bR; \qquad F\mapsto\phi_\sA^+\big(F(G)\big).
$$
Therefore, to prove the theorem, it suffices to show that $\widetilde{\tau}$ is a quasimorphism. We claim that the following inequality holds true for any $F_1,F_2\in\Aut(\sD)$:
$$
|\widetilde{\tau}(F_1F_2)-\widetilde{\tau}(F_1)-\widetilde{\tau}(F_2)|\leq6d.
$$
One can assume that $\widetilde{\tau}(F_1)=\widetilde{\tau}(F_2)=0$ by composing $F_1$ and $F_2$ with appropriate powers of the shift functor. By Lemma~\ref{lemma}, there is a sequence of exact triangles
$$
\xymatrix@C=.5em{
{} \ar@{}[r] & 0 \ar[rrrr] &&&& \star \ar[rrrr] \ar[dll] &&&& \star
\ar[rr] \ar[dll] && \cdots \ar[rr] && \star
\ar[rrrr] &&&& F_2G \ar[dll] \ar@{}[r] &   \\
&&& A_0 \ar@{-->}[ull] &&&& A_1[-1] \ar@{-->}[ull] &&&&&&&& A_{3d}[-3d] \ar@{-->}[ull] 
}
$$
for some objects $A_0,\ldots,A_{3d}$ in $\sA$.
By applying $F_1$ to the sequence, one obtains
$$
\widetilde{\tau}(F_1F_2)=
\phi_\sA^+(F_1F_2G)\leq\max\{\phi_\sA^+(F_1A_0),\phi_\sA^+(F_1A_1)-1,\ldots,\phi_\sA^+(F_1A_{3d})-3d\}\leq 3d
$$
and
\begin{align*}
\widetilde{\tau}(F_1F_2)=\phi_\sA^+(F_1F_2G) & \geq\phi_\sA^-(F_1F_2G) \\
&\geq\min\{\phi_\sA^-(F_1A_0),\phi_\sA^-(F_1A_1)-1,\ldots,\phi_\sA^-(F_1A_{3d})-3d\}\geq-6d.
\end{align*}
This concludes the proof.
\end{proof}

\begin{rmk}
Let $\hat{X}$ be the dual abelian variety of $X$. Then the product group $X\times\hat{X}$ is isomorphic to a normal subgroup of the group of autoequivalences $\Aut(\Db(X))$, where the elements of $X$ give translations $t_{x*}$, and the elements of $\hat{X}$ give tensor products $L\otimes-$ where $L\in\Pic^0(X)$ \cite{Orl}. Since these two types of functors do not effect $\phi^\pm_{\sA=\Coh(X)}$, the shifting numbers descend to the quotient group:
$$
\tau=\tau^\pm\colon\Aut(\Db(X))/(X\times\hat{X})\rightarrow\bR.
$$
There is an isomorphism 
$$
\Aut(\Db(X))/(X\times\hat{X})\cong\widetilde{U}(X\times\hat{X}),
$$
where $\widetilde{U}(X\times\hat{X})$ is a central extension of the group of certain isometric automorphisms $U(X\times\hat{X})$ by the cyclic group $\bZ$ \cite{Orl}.
It would be interesting to compare the quasimorphism on $\widetilde{U}(X\times\hat{X})$ defined by the shifting numbers of autoequivalences with the classical quasimorphisms associated to the central extensions of Lie groups of hermitian types \cite{BIW}.
Such relations have been established in \cite[Section~5]{FF} for abelian surfaces.
\end{rmk}

\section{Statements of conflict of interest and data availability}
The  author states that there is no conflict of interest.
Data sharing not applicable to this article as no datasets were generated or analysed during the current study.

\ContactInfo
\end{document}